\documentclass{amsart}

\usepackage{amsmath, amsthm, amssymb, mathrsfs, dsfont, fontenc}
\usepackage{tikz}

\usepackage[utf8]{inputenc}

\usepackage{xcolor}
\usetikzlibrary{scopes}

\usepackage[T1]{fontenc}

\usepackage[colorlinks=false,hidelinks]{hyperref}

\newtheorem{theorem}{Theorem}[section]

\newtheorem{corollary}[theorem]{Corollary}
\newtheorem{lemma}[theorem]{Lemma}
\newtheorem{remark}[theorem]{Remark}

\numberwithin{equation}{section}

\begin{document}

\title[Asai cube $L$-functions]{Asai cube $L$-functions and \\ the local Langlands correspondence}

\author[G. Henniart]{Guy Henniart}
\author[L. Lomel\'i]{Luis Lomel\'i}

\begin{abstract}
Let $F$ be a non-archimedean locally compact field. We study a class of Langlands-Shahidi pairs $({\bf H},{\bf L})$, consisting of a quasi-split connected reductive group $\bf H$ over $F$ and a Levi subgroup $\bf L$ which is closely related to a product of restriction of scalars of ${\rm GL}_1$'s or ${\rm GL}_2$'s. We prove the compatibility of the resulting local factors with the Langlands correspondence. In particular, let $E$ be a cubic separable extension of $F$. We consider a simply connected quasi-split semisimple group $\bf H$ over $F$ of type $D_4$, with triality corresponding to $E$, and let $\bf L$ be its Levi subgroup with derived group ${\rm Res}_{E/F} {\rm SL}_2$. In this way we obtain Asai cube local factors attached to irreducible smooth representations of ${\rm GL}_2(E)$; we prove that they are Weil-Deligne factors obtained via the local Langlands correspondence for ${\rm GL}_2(E)$ and tensor induction from $E$ to $F$. A consequence is that Asai cube $\gamma$- and $\varepsilon$-factors become stable under twists by highly ramified characters.
\end{abstract}

\maketitle

\section{Introduction}

\subsection{} Let $F$ be a non-archimedean locally compact field. Let $\bf G$ be a quasi-split connected reductive group over $F$, and let $r$ be a complex representation of the Langlands dual group ${}^{\rm L}G$ of $\bf G$. For interesting choices of $\bf G$ and $r$, the Langlands-Shahidi method attaches a $\gamma$-factor $\gamma(s,\pi,r,\psi)$ to a smooth non-trivial character $\psi$ of $F$ and an irreducible smooth generic representation $\pi$ of ${\bf G}(F)$. 

Let $\mathcal{W}_F$ denote the absolute Weil group of $F$, and $\mathcal{W}_F'$ its Weil-Deligne group. The local Langlands correspondence --still mostly conjectural-- associates to $\pi$ a Langlands parameter $\phi_\pi: \mathcal{W}_F' \rightarrow {}^{\rm L}G$. Whenever such a correspondence is established for $\bf G$, we should have the following main equality:
\begin{equation}\label{maineq}
   \gamma(s,\pi,r,\psi) = \gamma(s,r \circ \phi_\pi,\psi),
\end{equation}
where the $\gamma$-factors on the right are the ones defined by Langlands and Deligne.

\subsection{} A noteworthy example arises from a trio $({\bf G},{\bf H},{\bf L})$ of quasi-split groups such that: $\bf G$ is the restriction of scalars ${\rm Res}_{E/F}{\rm GL}_2$, where $E$ is a cubic separable extension of $F$, so that $G = {\bf G}(F) = {\rm GL}_2(E)$; $\bf H$ is an ambient quasi-split group (over $F$) semisimple and simply connected of type $D_4$, with triality corresponding to $E/F$ (see \S~\ref{AsaiL}); and $\bf L$ is the maximal Levi subgroup of $\bf H$ obtained by removing the central root $\alpha_2$ from the Dynkin diagram of $\bf H$:

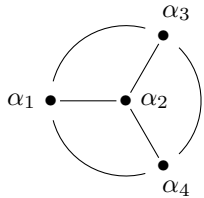
\begin{center}
\begin{tikzpicture}
   \draw[circle] (180:1) node {$\bullet$};
   \draw[circle] (180:1) node[left] {$\alpha_1$};
   \draw[circle] (195:1) arc (195:285:1);
   \draw[circle] (300:1) node {$\bullet$};
   \draw[circle] (300:1.35) node {$\alpha_4$};
   \draw[circle] (315:1) arc (315:405:1);
   \draw[circle] (60:1) node {$\bullet$};
   \draw[circle] (60:1.35) node {$\alpha_3$};
   \draw[circle] (75:1) arc (75:165:1);
   \draw[circle] (0:0) node {$\bullet$} node[right] {$\alpha_2$};
   \draw[circle] (60:.125) -- (60:.875);
   \draw[circle] (180:.125) -- (180:.875);
   \draw[circle] (300:.125) -- (300:.875);
\end{tikzpicture}
\end{center}
Writing $L = {\bf L}(F)$, there is a natural homomorphism of $L$-groups ${}^{\rm L}\iota : {}^{\rm L}G \rightarrow {}^{\rm L}L$, and a corresponding homomorphism $\iota : {\bf L} \rightarrow {\bf G}$. We use the Langlands-Shahidi method (see \cite{Sh1990} and \cite{LS}), applied to the pair $({\bf H},{\bf L})$ to produce $\gamma$-factors for $G$. More precisely, the adjoint action of ${}^{\rm L}L$ on the Lie algebra of the unipotent radical of the parabolic ${}^{\rm L}P$ of ${}^{\rm L}H$ with Levi ${}^{\rm L}L$ produces an $8$-dimensional representation $r_\mathcal{A}$ of ${}^{\rm L}L$. The Langlands-Shahidi method gives a $\gamma$-factor $\gamma(s,\tau,r_\mathcal{A},\psi)$ for any irreducible smooth generic representation $\tau$ of $L$. If $\pi$ is an irreducible smooth generic representation of $G$, we define Asai cube $\gamma$-factors for $\pi$ by setting
\begin{equation}\label{Asai:gamma:def}
   \gamma(s,\pi,{}^\otimes{\rm I},\psi) \mathrel{\mathop:}= \gamma(s,\tau,r_\mathcal{A},\psi),
\end{equation}
where $\tau$ is a (generic) irreducible component of $\pi \circ \iota$.

\subsection{} For ${\rm GL}_2$, the local Langlands correspondence was established in \cite{Ku}, see also the more detailed account of \cite{BuHe}. Let $\pi$ be as above, and let $\sigma$ be the corresponding degree $2$ representation of the Weil-Deligne group $\mathcal{W}_E'$ of $E$. Then, $\mathcal{W}_E'$ can be seen as an index $3$ subgroup of $\mathcal{W}_F'$, and we let ${}^\otimes{\rm I}(\sigma)$ be the $8$-dimensional representation of $\mathcal{W}_F'$ obtained from $\sigma$ by tensor induction from $\mathcal{W}_E'$ to $\mathcal{W}_F'$. The main equality in this case takes the form:
\begin{equation}\label{maineq:Asai}
   \gamma(s,\pi,{}^\otimes{\rm I},\psi) = \gamma(s,{}^\otimes{\rm I}(\sigma),\psi).
\end{equation}
We establish it in \S~\ref{proof:GL2class}. Prasad and Schulze-Pillot posed the question of equality between Langlands-Shahidi $\varepsilon$-factors with those on the Galois side \cite{PrSP2008}; in \S~\ref{consequences} we extend equality \eqref{maineq:Asai} to non-generic $\pi$ and derive equality for $L$-functions and $\varepsilon$-factors:
\begin{align}\label{maineq:L3}
   L(s,\pi,{}^\otimes{\rm I}) &= L(s,{}^\otimes{\rm I}(\sigma)) \\ \nonumber
   \varepsilon(s,\pi,{}^\otimes{\rm I},\psi) &= \varepsilon(s,{}^\otimes{\rm I}(\sigma),\psi).
\end{align}
As a consequence of \eqref{maineq:Asai}, we also prove stability of Asai cube $\gamma$-factors after twisting by a suitably highly ramified character in \S~\ref{stability:prop}.

\subsection{Remark} An analogous $\gamma$-factor can be defined via the integral representation of Garrett \cite{Ga1987}, and also Piatetski-Shapiro and Rallis \cite{PSRa1987}. This approach can be very helpful in determining the location of poles and non-zero regions, see the work of Ikeda, e.g. \cite{Ik1992}. We do not address the question of equality of those factors with ours.

\subsection{} The proof of  \eqref{maineq:Asai} is local-global in nature, and its principle is well known: by the multiplicativity property of $\gamma$-factors, one reduces to the case where $\pi$ is supercuspidal. And if so, one uses instances where the global Langlands correspondence for ${\rm GL}_2$ has been established \cite{JaLa, La, Tu1978} --more details in \S~\ref{sketch:proof} and a full proof in \S~\ref{proof:GL2class}.

The idea of a local to global argument goes back to the early days of class field theory; in the context of Langlands correspondences it has been used repeatedly since the 1970's. In situations closely related to ours, it was used by H. Kim in \cite{Ki2002} to establish the main equality \eqref{maineq} for a triple product of representations of ${\rm GL}_2(F)$, and by M. Krishnamurthy in \cite{Kr2003} for a product $\pi \times \Pi$, where $\pi$ is a representation of ${\rm GL}_2(F)$ and $\Pi$ of ${\rm GL}_2(E)$, $E/F$ quadratic. Specifically, we use the remark by Weil that local Galois groups are solvable \cite{We1974}; that idea was used again by Ramakrishnan in \cite{Ra2000}. A local-global method was used when $F$ has positive characteristic in recent work of the authors \cite{GaLo, HeLo2011, HeLo2013a, HeLo2013b}.

\subsection{} In the global situation, we use a cubic separable extension $l/k$ of global fields, and we have to consider the algebra $k_v \otimes_k l$ for all places $v$ of $k$. Since such an algebra is not necessarily a field, we are forced to consider other $\gamma$-factors than the Asai cube ones, in particular, those of \cite{Ki2002} and \cite{Kr2003} mentioned above. Rather than relying on those specific instances, however, we prove in one swoop a more general result. Namely, Theorem~\ref{GL2class:thm} below, which illustrates the full extent of the method.

We say that a connected reductive group $\bf G$ is of ${\rm GL}_n$-\emph{kind} if ${\bf G}$ is isomorphic to a product ${\bf G}_1 \times \cdots \times {\bf G}_t$, where for each $i$ we have ${\bf G}_i = {\rm Res}_{l_i/k}{\rm GL}_{n_i}$, $n_i \leq n$. Here, ${\rm Res}_{l_i/k}{\rm GL}_{n_i}$ denotes restriction of scalars with respect to a separable extension $l_i$ over a base field $k$.

The Langlands-Shahidi method applies to a pair of quasi-split connected reductive groups $({\bf H},{\bf L})$ such that there is a parabolic subgroup $\bf P$ of the ambient group $\bf H$, with Levi component $\bf L$. The unipotent radical of $\bf P$ is denoted by $\bf N$ and has Lie algebra $\mathfrak{n}$. Let $\rho$ be an irreducible component of the adjoint representation of ${}^{\rm L}L$ on ${}^{\rm L}\mathfrak{n}$. We then say that a representation $r$ of ${}^{\rm L}G$ is an \emph{LS-representation} if there is such a pair $({\bf H},{\bf L})$ and an $L$-map ${}^{\rm L}\iota : {}^{\rm L}{\bf G} \rightarrow {}^{\rm L}{\bf L}$, inducing a separable isogeny on derived groups, such that $r = \rho \circ {}^{\rm L}\iota$. Then, the Langlands-Shahidi method attaches a $\gamma$-factor $\gamma(s,\pi,r,\psi)$ to any generic irreducible smooth representation $\pi$ of $G$ --the $\gamma$-factor depends only on $r$, not on the way $r$ is written $\rho \circ {}^{\rm L}\iota$ for some $\iota$.

\begin{theorem}\label{GL2class:thm}
Let $r$ be an LS-representation of ${}^{\rm L}G$ for a group $\bf G$ of ${\rm GL}_2$-kind. Let $\pi$ be a smooth irreducible generic representation of $G$, and let $\phi_\pi$ be its Langlands parameter. Let $\psi : F \rightarrow \mathbb{C}^\times$ be a non-trivial character. Then
\begin{equation*}
   \gamma(s,\pi,r,\psi) = \gamma(s,r \circ \phi_\pi,\psi).
\end{equation*}
\end{theorem}

Equality \eqref{maineq:Asai} is a special case. The proof of Theorem~\ref{GL2class:thm} occurs in \S~\ref{proof:GL2class}; its application to the Asai cube factors is found in \S~\ref{AsaiL}.

\subsection{Remark} Compatibility with the local Langlands correspondence, i.e. equality \eqref{maineq}, can be stated for all LS-representations of groups of ${\rm GL}_n$-kind, since the local Langlands correspondence is known for ${\rm GL}_n$ \cite{La1973, LaRaSt1993, HaTa2001, He2000, Sc2013}. Equality \eqref{maineq} holds for the Rankin-Selberg $\gamma$-factors $\gamma(s,\pi_1\times\pi_2,\psi)$, essentially by definition of the local Langlands correspondence. Another case occurs when $G = {\rm GL}_n(F)$ and $r$ is the exterior square or the symmetric square representation of ${}^{\rm L}G$. In that case, for a local field $F$ of characteristic zero, \eqref{maineq} was proved up to a root of unity in \cite{He2010} and equality in \cite{CoShTs2017}; when $F$ has positive characteristic, exact equality was proved in \cite{HeLo2011} (and twisted exterior square or symmetric square $\gamma$-factors, for $G={\rm GL}_n(F) \times F^\times$, were treated in \cite{GaLo2015}). Sill another case is when $G = {\rm GL}_n(E)$, $E/F$ quadratic separable, and $r$ gives quadratic Asai $\gamma$-factors: for $F$ of characteristic zero, see \cite{He2010} for equality up to a root of unity and work in progress of D. Shankman, adapting \cite{CoShTs2017} for true equality; for $F$ of characteristic $p$, \cite{HeLo2013b} establishes equality \eqref{maineq}.

In fact, when $F$ has positive characteristic, the method of \cite{HeLo2011,HeLo2013a,HeLo2013b} was generalized by Gan and Lomel\'i in \cite{GaLo}, yielding \eqref{maineq} for LS representations of ${\rm GL}_n$-type; so when $F$ has positive characteristic, Theorem~\ref{GL2class:thm} is in fact a consequence of \cite{GaLo}. Nonetheless, in \S~\ref{p>0} we indicate how to adapt \cite{HeLo2011} and \cite{HeLo2013b} to the present situation, as it gives a more direct proof than \cite{GaLo}.

\subsection{}\label{sketch:proof} For the benefit of the reader, we now explain the local-global argument in the proof of \eqref{maineq:Asai} (for the Asai cube factors), when $F$ has characteristic zero. The general case of Theorem~\ref{GL2class:thm} is fully proved in \S~\ref{proof:GL2class}, using a similar pattern.

Let $k$ be a number field. The principal obstacle is the lack of a complete Langlands correspondence for ${\rm GL}(2)$ over $k$, so we need to use the partial results available in that direction. More precisely, as in the positive characteristic case, we use a list of properties of the $\gamma$-factors $\gamma(s,\pi,r,\psi)$. All but one of these properties are local, and the corresponding properties for $\gamma(s,r \circ \phi_\pi,\psi)$ are easy to prove. One of these properties, namely multiplicativity, concerns the case where $\pi$ is a component of a representation ${\rm Ind}_P^G \rho$ induced from a proper parabolic subgroup of $G$. Multiplicativity and an inductive argument help us reduce the proof of \eqref{maineq:Asai} to the case where $\pi$ is cuspidal.

To treat the cuspidal case, our tool is the global functional equation, which we now describe in the Asai cube case. Let $l$ be a cubic extension of $k$, and $\Pi$ a cuspidal automorphic representation of ${\bf G}(\mathbb{A}_k) = {\rm GL}_2(\mathbb{A}_l)$ where ${\bf G} = {\rm Res}_{l/k}{\rm GL}_2$. Then for each place $v$ of $k$, $\Pi_v$ is a generic irreducible representation of $G_v = {\bf G}(k_v) = {\rm GL}_2(k_v \otimes_k l)$, so that the Langlands-Shahidi method produces $\gamma$-factors $\gamma(s,\Pi_v,{}^\otimes{\rm I},\Psi_v)$ for any non-trivial character $\Psi$ of $k \backslash \mathbb{A}_k$. The global functional equation is
\begin{equation}\label{globalFE}
   L^S(s,\pi,{}^\otimes{\rm I}) = \prod_{v \in S} 
   \gamma(s,\Pi_v,{}^\otimes{\rm I},\Psi_v) 
   L^S(1-s,\widetilde{\Pi},{}^\otimes{\rm I}),
\end{equation}
where $S$ is a finite set of places of $k$, containing the infinite ones, such that $l/k$, $\Pi$ and $\Psi$ are unramified for $v \notin S$, and $L^S$ denotes the product of (unramified) local $L$-functions ranging over all $v\notin S$ --which by construction will be compatible with Galois $L$-functions. Also, $\widetilde{\Pi}$ is the representation contragredient to $\Pi$.

The proof of \eqref{maineq:Asai} then proceeds as follows for $\pi$ cuspidal. We choose a cubic extension of number fields $l/k$, with a place $v_0$ giving the local extension $E/F$, and a cuspidal automorphic representation $\Pi$ as above, giving $\pi$ at $v_0$. We assume that $\psi$ is the component at $v_0$ of a global additive character $\Psi$ (local properties later allow us to derive the case where $\psi$ is arbitrary). We apply the functional equation to $\Pi$. The set $S$ not only has to contain $v_0$, but we also need to adjust the global situation so that at places in $S$, different from $v_0$, equality \eqref{maineq:Asai} is valid. The case of Archimedean places comes from Shahidi \cite{Sh1985,Sh1990}, but at a finite place $v \in S \setminus \left\{ v_0 \right\}$ we need to assume that $\Pi_v$ is somewhat less complicated than $\pi$, for example, $\Pi_v$ not cuspidal or a level zero cuspidal. As mentioned before, that procedure necessarily yields situations which are not giving rise to Asai cube factors. For example, at a finite place distinct from $v_0$ the extension $l/k$ may be split or partially split; the former case requires a triple Rankin-Selberg product and the latter a quadratic Asai $\gamma$-factor. This explains why we consider the more general result of Theorem~\ref{GL2class:thm}, which also has intrinsic interest.

The main problem, however, is to get a corresponding functional equation on the Galois side. While over function fields the functional equation for $L$-functions of $\ell$-adic Galois representations is a consequence of results of Deligne and Laumon, over number fields it is available only for complex representations of Weil groups. Only if $\Pi$ corresponds to a complex representation $\Sigma$ of the Weil group $\mathcal{W}_l$ of $l$, can we compare the functional equation \eqref{globalFE} with the one for $\Sigma$:
\begin{equation}\label{galoisFE}
   L^S(s,{}^\otimes{\rm I}(\Sigma)) = \prod_{v \in S} \gamma(s,{}^\otimes{\rm I}(\Sigma_v),\Psi_v) L^S(1-s,{}^\otimes{\rm I}(\widetilde{\Sigma})).
\end{equation}
The comparison is done term by term: the $L$-functions in \eqref{globalFE} and \eqref{galoisFE} are the same because $\Pi$ corresponds to $\Sigma$, and the $\gamma$-factors at places in $S$ distinct from $v_0$ are the same by the inductive argument alluded to above. Equality \eqref{maineq:Asai} for $\pi$ then follows.

Our way to produce corresponding $\Pi$ and $\Sigma$ is to start from the Weil group side, with $l/k$ chosen as mentioned above giving $E/F$ at $v_0$. The representation $\sigma$, corresponding to $\pi$ cuspidal, is irreducible, hence has finite image in the projective linear group; this image is isomorphic to $A_4$ or $S_4$, or is dihedral. In each case, we produce a representation $\Sigma$ of $\mathcal{W}_l$, yielding $\sigma$ at the place $v_0$, and with the same image as $\sigma$ in the projective linear group. By the strong Artin conjecture of Langlands and Tunnell \cite{La,Tu1981}, there is an automorphic cuspidal representation $\Pi$ of ${\rm GL}_2(\mathbb{A}_l)$ corresponding to $\Sigma$, and we can proceed as sketched.

\subsection*{Acknowledgments} The authors would like to thank D. Prasad and F. Shahidi for mathematical discussions. The first author thanks his home institutions, Universit\'e Paris-Saclay and CNRS; in addition to Pontificia Universidad Cat\'olica de Valpara\'iso for a visit (supported by MEC Grant 80170039) where a second version of the paper was written. The second author would like to thank IH\'ES and Universit\'e d'Orsay for their hospitality during visits in 2016; in addition to IISER, Pune and TIFR, Mumbay during January and February 2017; he was supported in part by Project VRIEA/PUCV 039.367 and FONDECYT Grant 1171583.

\section{LS method and ${\rm GL}(2)$}\label{LSmethod}

For the moment, let $F$ denote either a local field or a global one.  We write $\mathfrak{p}$ to denote the characteristic of $F$, where we allow $\mathfrak{p}$ to be either a prime $p$ or $0$. Fix $\mathfrak{p}$. As mentioned in the introduction, we say that a connected reductive group ${\bf G}/F$ is of ${\rm GL}_n$-\emph{kind} if ${\bf G}$ is isomorphic to a product ${\bf G}_1 \times \cdots \times {\bf G}_t$, where we have that each ${\bf G}_i$ is a group ${\rm Res}_{E_i/F}{\rm GL}_{n_i}$, $n_i \leq n$, obtained via restriction of scalars with respect to a separable extension $E_i$ over the base field $F$. Given an algebraic group $\bf H$ defined over $F$, we write $H={\bf H}(F)$. We write ${\bf Z}_{\bf H}$ to denote the center of $\bf H$.

We let $\bf H$ be an ambient group that is quasi-split connected reductive defined over $F$, together with a subgroup $\bf L$ which is the Levi component of a maximal parabolic subgroup $\bf P$ of $\bf H$. Let $\bf N$ be the unipotent radical of $\bf P$ with Lie algebra $\mathfrak{n}$. The Langlands-Shahidi method \cite{Sh1990,LS}, applies to a representation $\rho$ which is an irreducible component of the adjoint representation of ${}^{\rm L}L$ on ${}^{\rm L}\mathfrak{n}$. We work with an extended version of these representations for groups $\bf G$ of ${\rm GL}_2$-kind.

An \emph{LS-representation} $r$ of ${}^LG$ is one obtained via a pair $({\bf H},{\bf L})$ of quasi-split connected reductive groups and a representation $\rho$ as above, together with a map of L-groups ${}^{\rm L}\iota : {}^{\rm L}{\bf G} \rightarrow {}^{\rm L}{\bf L}$, inducing a separable isogeny on derived groups, and such that $r = \rho \circ {}^{\rm L}\iota$.

\subsection{Notation} From \S~\ref{LSmethod} to \S~\ref{consequences} of the article, we let $F$ denote a local field and $k$ a global one.

Let $\mathscr{A}_n(\mathfrak{p})$ be the class of tuples $(F,\psi,{\bf G},\pi,r)$ where $F$ is a locally compact field of characteristic $\mathfrak{p}$, $\psi$ is a non-trivial additive character of $F$, a quasi-split reductive group $\bf G$ of ${\rm GL}_n$-kind, $\pi$ a smooth irreducible representation $\pi$ of $G = {\bf G}(F)$, and $r$ an LS-representation of ${}^{\rm L}G$. We say that a tuple $(F,\psi,{\bf G},\pi,r) \in \mathscr{A}_n$ is generic if $\pi$ is generic.

If $\pi$ is a smooth irreducible representation of $G$, we can adapt Silberger's reasoning \cite{Si1979}, in order to show that $\pi$ reduces in $L$ to a sum of finitely many irreducibles. Then, the Langlands-Shahidi method attaches a $\gamma$-factor
\begin{equation}\label{general:def}
   \gamma(s,\pi,r,\psi) \mathrel{\mathop:}= \gamma(s,\tau,r,\psi) \text{ for } (F,\psi,{\bf G},\pi,\psi) \in \mathscr{A}_2(\mathfrak{p}),
\end{equation}
where $\tau$ is the (unique) generic component of the restricition of $\pi$ to $L$.

We also let $\mathscr{A}_n^{\rm glob}(\mathfrak{p})$ be the class of tuples $(k,\Psi,{\bf G},\Pi,r,S)$ where $k$ is a global field of characteristic $\mathfrak{p}$, $\Psi$ a non-trivial character of $\mathbb{A}_k/k$ (note that $\Psi_v$ is unramified almost everywhere), $\bf G$ a quasi-split reductive group over $k$ of ${\rm GL}_n$-kind, $\Pi$ a cuspidal automorphic representation of ${\bf G}(\mathbb{A}_k)$ (we recall that such representations are globally generic), $r$ is an LS-representation of ${}^{\rm L}G$, and $S$ a finite set of places of $k$ containing the Archimedean ones, such that ${\bf G}_v$, $\Pi_v$ and $\Psi_v$ are unramified for $v \notin S$. Clearly such a tuple gives, for each place $v$ of $k$, a local tuple $(k_v,\Psi_v,{\bf G}_v,\Pi_v,r_v)$ where $r_v$ is obtained by composing $r$ with the natural map ${}^{\rm L}G_v \rightarrow {}^{\rm L}G_k$.

\subsection{Properties of LS $\gamma$-factors}\label{LS:axioms}

\begin{enumerate}
\item[(i)] \emph{(Naturality).} Let $(F,\psi,{\bf G},\pi,r) \in \mathscr{A}_2(\mathfrak{p})$ be generic. Let $\eta: F' \rightarrow F$ be an isomorphism extending to $G' = {\bf G}(F') \cong {\bf G}(F) = G$. Let $(F',\psi',{\bf G},\pi',r) \in \mathscr{A}_2(\mathfrak{p})$ be the tuple obtained via $\eta$. Then
   \begin{equation*}
      \gamma(s,\pi,r,\psi) = \gamma(s,\pi',r,\psi').
   \end{equation*}
\item[(ii)] \emph{(Isomorphism)}. Let $(F,\psi,{\bf G},\pi_j,r) \in \mathscr{A}_2(\mathfrak{p})$, $j=1$, $2$, be generic. If $\pi_1 \cong \pi_2$, then
   \begin{equation*}
      \gamma(s,\pi_1,r,\psi) = \gamma(s,\pi_2,r,\psi).
   \end{equation*}
\item[(iii)] \emph{(Dependence on $\psi$).} Let $(F,\psi,{\bf G},\pi,r) \in \mathscr{A}_2(\mathfrak{p})$ be generic. Given $a \in F^\times$, let $\psi^a : F \rightarrow \mathbb{C}^\times$ be the character given by $\psi^a(x) = \psi(ax)$. Then there is a rational character $z: F^\times \rightarrow Z_G$, $a \mapsto z(a) \in Z_G$, such that
   \begin{equation*}
      \gamma(s,\pi,r,\psi^a) = \omega_{\pi}(z(a)) \left| a \right|_F^{\dim(r)(s-\frac{1}{2})} \cdot \gamma(s,\pi,r,\psi).
   \end{equation*}
\item[(iv)] \emph{(Multiplicativity).} Let $(F,\psi,{\bf G},\pi,r) \in \mathscr{A}_2(\mathfrak{p})$ be generic. Let $\bf P$ be a parabolic subgroup of $\bf G$, with Levi subgroup $\bf M$, and assume that $\pi$ is the unique generic component of
    \begin{equation*}
      {\rm ind}_P^G \rho,
   \end{equation*}
where $\rho$ is a generic smooth representation of $M$. Then $\bf M$ is of ${\rm GL}_2$-kind, the restriction of $r$ to ${}^{\rm L}M$  is a sum of LS representations $r_j$, $j \in \mathscr{J}$, and
\begin{align*}
   \gamma(s,\pi,r,\psi) = \prod_{j \in \mathscr{J}} \gamma(s,\rho,r_j,\psi).
\end{align*}
\item[(v)] \emph{(Compatibility at representations with Iwahori fixed vectors).} Assume $F$ is non-archimedean and let $(F,\psi,{\bf G},\pi,r) \in \mathscr{A}_2(\mathfrak{p})$ be generic, such that $\pi$ has fixed vectors under Iwahori subgroups of $G$. Let $\phi_\pi : \mathcal{W}_F' \rightarrow {}^{\rm L}G$ be the Langlands parameter corresponding to $\pi$. Then
   \begin{equation*}
      \gamma(s,\pi,r,\psi) = \gamma(s,r \circ \phi_\pi,\psi).
   \end{equation*}
\item[(vi)] \emph{(Compatibility at archimedean places).} Assume that $F$ is Archimedean and such that $(F,\psi,{\bf G},\pi,r) \in \mathscr{A}_2(\mathfrak{p})$ is generic. Let $\phi_\pi : \mathcal{W}_F \rightarrow {}^{\rm L}G$ be the Langlands parameter corresponding to $\pi$. Then
   \begin{equation*}
      \gamma(s,\pi,r,\psi) = \gamma(s,r \circ \phi_\pi,\psi).
   \end{equation*}
\end{enumerate}

Compatibility at archimedean places is the subject of \cite{Sh1985}. There is one global property that connects the local theory of $\gamma$-factors with partial $L$-functions defined by
\begin{equation*}
   L^S(s,\Pi,r) = \prod_{v \notin S} L(s,\Pi_v,r_v),
\end{equation*}
for $(k,\Psi,{\bf G},\Pi,r,S) \in \mathscr{A}_2(\mathfrak{p})^{\rm glob}$.

\begin{enumerate}
\item[(vii)] \emph{(Functional equation).} Let $(k,\Psi,{\bf G},\Pi,r,S) \in \mathscr{A}_2(\mathfrak{p})^{\rm glob}$. Then
   \begin{equation*}
      L^S(s,\Pi,r) = \prod_{v \in S} \gamma(s,\Pi_v,r,\Psi_v) L^S(1-s,\widetilde{\Pi},r).
   \end{equation*}
\end{enumerate}

In \cite{HeLoFuture} we adapt Theorem~3.5 of \cite{Sh1990} and Theorem~4.1 of \cite{LS}, which give the existence and uniqueness of a system of $\gamma$-factors on $\mathscr{A}_2(\mathfrak{p})$ with Properties (i)--(vii).

\subsection{Properties of the corresponding Galois $\gamma$-factors}\label{Galois:axioms}
Let $(F,\psi,{\bf G},\pi,r)$ be a local tuple, and let $\phi_\pi$ be the $L$-parameter attached to $\pi$ via the local Langlands correspondence. Galois $\gamma$-factors $\gamma(s,r \circ \phi_\pi,\psi)$ possess the analogous properties of naturality ${\rm (i)}_\mathscr{G}$, and isomorphism ${\rm (ii)}_\mathscr{G}$ --in fact, when we talk about ``the'' $L$-parameter attached to $\pi$, we rather mean its equivalence class. Dependence on $\psi$ is, on the outset, more precise than for $\gamma(s,\pi,r,\psi)$:
\begin{equation*}
   {\rm (iii)}_\mathscr{G} \ \gamma(s,r \circ \phi_\pi,\psi^a) = \det(r \circ \phi_\pi)(a) \left| a \right|_F^{\dim(r)(s-\frac{1}{2})} \cdot \gamma(s,r \circ \phi_\pi,\psi),
\end{equation*}
where $\det(r \circ \phi_\pi)$ is seen as a character of $F^\times$ via class field theory.  And, multiplicativity is exactly parallel to the same property for $\gamma(s,\pi,r,\psi)$:

\begin{enumerate}
\item[${\rm (iv)}_\mathscr{G}$] With the hypothesis of Property~${\rm (iv)}$ of \S~\ref{LS:axioms}, the $L$-parameter $\phi_\pi$ factorizes as $\phi_\pi \circ i$, where $i$ is the inclusion of ${}^{\rm L}M$ into ${}^{\rm L}G$, and then
\begin{align*}
   \gamma(s,r \circ \phi_\pi,\psi) = \prod_{j \in \mathscr{J}} \gamma(s,r_j \circ \phi_\pi,\psi).
\end{align*}
\end{enumerate}
Properties ${\rm (v)}$ and ${\rm (vi)}$ of \S~3.2 are already statements about equality between LS and Galois $\gamma$-factors.

For the final property, we consider a global tuple $(k,\Psi,{\bf G},\Pi,r,S)$, and we \emph{assume} that there is a global Weil group $L$-parameter
\begin{equation*}
   \Phi_\Pi: \mathcal{W}_k \rightarrow {}^{\rm L}G,
\end{equation*}
which corresponds to $\Pi$ (at all places); it then satisfies the functional equation
\begin{equation*}
   {\rm (vii)}_\mathscr{G} \
   L^S(s,r \circ \Phi_\Pi) = \prod_{v \in S} \gamma(s,r \circ \Phi_{\Pi_v},\Psi_v) L^S(1-s,r \circ \Phi_{\widetilde{\Pi}}),
\end{equation*}   
where partial $L$-functions are given by
\begin{equation*}
   L^S(s,\Phi_\Pi,r) = \prod_{v \notin S} L(s,\Phi_{\Pi_v},r_v).
\end{equation*}
Our main point in the proof of Theorem~\ref{GL2class:thm} will be to ensure this assumption.

\section{Proof of Theorem~\ref{GL2class:thm}}\label{proof:GL2class}

\subsection{}\label{123} In this section, we prove Theorem~\ref{GL2class:thm} for $\mathfrak{p}=0$ in \S\S~\ref{begin:p0}--\ref{nonD:0}, then for $\mathfrak{p}>0$ in \S~\ref{p>0}. But first, we gather arguments which are the same in both situations. Let $F$ be non-archimedean, and ${\bf G}/F$ a connected reductive group of ${\rm GL}_2$-kind.

\medskip

\noindent{\bf 1)} Let $\pi$ be a smooth irreducible generic representation of $G$. Then there is a parabolic subgroup $\bf P$ of $\bf G$ with Levi subgroup $\bf M$ and a smooth irreducible generic supercuspidal representation $\rho$ of $M$ such that $\pi$ is a component of ${\rm Ind}_P^G(\rho)$. By the multiplicativity properties, \S~\ref{LS:axioms}(iv) and \S~\ref{Galois:axioms}${\rm (iv)}_\mathscr{G}$, if Theorem~\ref{GL2class:thm} is valid for the tuples $(F,\psi,{\bf M},\rho,r_j)$, then it is valid for $(F,\psi,{\bf G},\pi,r)$.

We shall thus reason by induction on $n_G = \dim{\bf G} - \dim{\bf C}$, where $\bf C$ is the maximal $F$-split torus in the center of $\bf G$. If $\bf M$ is a proper Levi subgroup of $\bf G$ then $n_G > n_M$, so to prove Theorem~\ref{GL2class:thm} we may assume that $\pi$ is supercuspidal.

\medskip

\noindent{\bf 2)} Choose a global field $k$ with a place $v$ and an isomorphism $\eta$ of $k_v$ onto $F$. Assume Theorem~\ref{GL2class:thm} is proved for tuples $(F,\psi,{\bf G},\pi,r)$ such that $\psi \circ \eta$ is the component at $v$ of a character of $\mathbb{A}_k/k$. Choose such a tuple, and a character $\Psi$ of $\mathbb{A}_k/k$ with $\Psi_v = \psi \circ \eta$.

For $b \in k^\times$, $\psi^{\eta(b)}\circ\eta$ is the component at $v$ of $\Psi^b : x \mapsto \Psi(bx)$, so we get by our assumption
\begin{equation*}
   \gamma(s,\pi,r,\psi^{\eta(b)}) = \gamma(s,r\circ\phi_\pi,\psi^{\eta(b)}).
\end{equation*}
Then, by applying \S~\ref{LS:axioms}(iii) and \S~\ref{Galois:axioms}${\rm (iii)}_\mathscr{G}$, we conclude that
\begin{equation*}
   \omega_\pi(z(\eta(b)) = \det(r\circ\phi_\pi)(\eta(b)),
\end{equation*}
for $b \in k^\times$ and $z : F^\times \rightarrow Z_G$ the rational character of \S~\ref{LS:axioms}(iii). But both $\omega_\pi \circ z$ and $\det(r \circ \phi_\pi)$ are continuous characters of $F^\times$. Since they are equal on $\eta(k^\times)$, which is dense in $F^\times$, they are equal. Thus, for all $a \in F^\times$ we get
\begin{equation*}
   \gamma(s,\pi,r,\psi^a) = \gamma(s,r\circ\phi_\pi,\psi^a).
\end{equation*}

\noindent{\bf 3)} To apply {\bf 2)} we first need to lift our tuple $(F,\psi,{\bf G},\pi,r)$ to a global one. The following lemma (Lemma~3.6 of \cite{He1983}, Lemma~4.13 of \cite{DeAntwerp}) will also be used later.

\begin{lemma}\label{GF:lemma}
Let $\widetilde{E}/F$ be a finite Galois extension. Then there exist a global field $k$, a finite Galois extension $l/k$, with a place $v_0$ of $k$, and an isomorphism $\eta$ of $k_{v_0}$ onto of $F$ inducing an isomorphism of $k_{v_0} \otimes_k l$ onto $\widetilde{E}$. In particular, $l$ has only one place $w_0$ above $v_0$ and the decomposition subgroup of ${\rm Gal}(l/k)$ at $w_0$ is itself, and identifies with ${\rm Gal}(\widetilde{E}/F)$ via $\eta$.
\end{lemma}

\begin{remark}\label{rmk:0p} {\rm (i)} Assume that $F$ is a finite extension of $\mathbb{Q}_p$ (hence $\mathfrak{p}=0$). Applying the above procedure to an $\widetilde{E}$ which is Galois over $\mathbb{Q}_p$, we find a number field $l$ with only one place above $p$. {\rm (ii)} When $F$ has characteristic $2$ (so $\mathfrak{p}=2$), we shall in fact use a stronger result of Gabber-Katz \cite{Ka1986}.
\end{remark}

Using Lemma~\ref{GF:lemma}, we can lift ${\bf G}/F$ and its LS representation $r$ to a number field $k$. Let us briefly provide the argument for this (it applies to a general situation, not necessarily of ${\rm GL}_2$-kind \cite{HeLoFuture}). Recall that $r = \rho \circ {}^{\rm L}\iota$ has an underlying pair of connected quasi-split reductive groups $({\bf H},{\bf L})$ defined over $F$, together with an L-group homomorphism ${}^{\rm L}\iota : {}^{\rm L}G \rightarrow {}^{\rm L}L$.

We choose a large enough finite Galois extension $\widetilde{E}$ of $F$ in $\overline{F}$, such that $\bf G$, $\bf H$ and $\bf L$ are split over $\widetilde{E}$. The action of $\mathcal{W}_F$ on the based root data of the (pinned) groups $\bf G$, $\bf H$, $\bf L$ factors through ${\rm Gal}(\widetilde{E}/F)$. Using the lemma we choose a number field $k$, and extension $l$, getting at the place $v_0$ an isomorphism of ${\rm Gal}(l/k)$ onto ${\rm Gal}(\widetilde{E}/F)$; in particular, we obtain based root data with an action of ${\rm Gal}(l/k)$, determining (uniquely up to unique isomorphism) pinned groups ${\bf G}_k$, ${\bf H}_k$, ${\bf L}_k$ together with an L-homomorphism ${}^{\rm L}\iota_k : {}^{\rm L}G_k \rightarrow {}^{\rm L}L_k$. As the action of ${}^{\rm L}L$ on ${}^{\rm L}\mathfrak{n}$ can be seen on the root datum, $\rho$ corresponds to an irreducible representation $\rho_k$ of ${}^{\rm L}L_k$ and we put $r_k = \rho_k \circ {}^{\rm L}\iota_k$. At the place $v_0$, the isomorphism $l_{v_0} \simeq \widetilde{E}$ gives unique isomorphisms of pinned groups $G_k \otimes k_{v_0} \simeq G$, compatible with the based root data, and similarly for $H$ and $L$. On the L-group side, if we choose a separable algebraic closure $\bar{k}$ of $k$ containing $l$, we get an L-group ${}^{\rm L}G_k$, and an embedding ${}^{\rm L}G \rightarrow {}^{\rm L}G_k$ at the place $v_0$, and similarly for ${}^{\rm L}H$ and ${}^{\rm L}L$. Composing $r_k$ with the embedding ${}^{\rm L}G \rightarrow {}^{\rm L}G_k$, we retrieve $r$ back.

Of course, to only lift the group ${\bf G}/F$ of ${\rm GL}_2$-type, it is easier. Indeed, if $\bf G$ is the product of ${\rm Res}_{E_i/F}{\rm GL}_{n_i}$, then $\widetilde{E}$ contains the $E_i$'s, and we can take ${\bf G}_k$ to be the product of ${\rm Res}_{l_i/k}{\rm GL}_{n_i}$, where $l_i$ is the sub-extension of $l$ corresponding to $E_i$.

\subsection{}\label{lemma2:ss} Before continuing, we shall need another well-known lemma. We include a proof for the convenience of the reader, where we use the following notation for a global field $k$: {\bf (i)} if $\mathfrak{p}=0$, interpret $v \in \left\{ \infty \right\}$ to mean that $v$ is an archimedean place of $k$; {\bf (ii)} if $\mathfrak{p}>0$, we take $\infty$ to be a fixed place at infinity for $k$, $v_0 \neq \infty$.

\begin{lemma}\label{GL1:glob}
Let $k$ be a global field, $v_0$ a finite place of $k$, $\lambda$ a multiplicative character of $k_{v_0}^\times$. Then there exists a character $\chi = \otimes \chi_v$ of $\mathbb{A}_k/k^\times$ such that:
\begin{enumerate}
   \item[(i)] $\chi_{v_0} \simeq \lambda$;
   \item[(ii)] $\chi_v$ is unramified for $v \notin \left\{ v_0, \infty \right\}$;
   \item[(iii)] if $\mathfrak{p}>0$, then $\chi_\infty$ is tame.
\end{enumerate}
\end{lemma}
\begin{proof}
First let $\mathfrak{p}=0$. For each finite place $v$ of $k$, the group of units $U_v$ of $k_v^\times$ is compact and so is the product
\begin{equation*}
   U = \prod_{v\ {\rm finite}} U_v.
\end{equation*}
We have $k^\times \cap U = \left\{ 1 \right\}$ in $\mathbb{A}_k^\times$, and $k^\times U$ is closed in $\mathbb{A}_k^\times$. The character of $k^\times U$ which is trivial on $k^\times$ and $U_v$ for $v \neq v_0$, and given by $\lambda$ on $U_{v_0}$, extends to a (unitary) character $\chi'$ of $\mathbb{A}_k^\times$, trivial on $k$ and $U_v$ for $v \neq v_0$, and $\chi_{v_0}'$ differs from $\lambda$ by a power of the absolute value character of $k_{v_0}^\times$. Modifying $\chi'$ accordingly by a power of the absolute value id\`ele class character of $\mathbb{A}_k^\times$, we find the desired $\chi$.

\medskip

Now assume $\mathfrak{p}>0$. For a place $v$ of $k$, let $U_v^1$ be the pro-$p$ radical of $U_v \subset k_v^\times$. The groups
\begin{equation*}
   U' = U_{\infty}^1 \times \prod_{v \notin \{ v_0, \infty\}} U_v
   \text{ and } U = U_{v_0} \times U'
\end{equation*}
are compact and we have a map $\Phi : U \rightarrow k^\times \backslash \mathbb{A}_k^\times$, with ${\rm Im}\,\Phi$ compact. Also, ${\rm Ker}\,\Phi = \{ 1 \}$. Indeed, let $x \in {\rm Ker}\,\Phi$, so that $x_v \in U_v$ for all places $v$ of $k$. Then $x$ must be a constant function on the smooth projective curve $X$ with function field $k$. But, $x_\infty \in 1+\mathfrak{p}_\infty$, gives $x_\infty=1$. Hence, $x=1$.

The character
\begin{equation*}
   \lambda\vert_{U_{v_0}} \otimes \mathds{1}_{U'} \text{ of } 
   U,
\end{equation*}
gives a character of ${\rm Im}\,\Phi$, which is a closed subgroup of $k^\times \backslash\mathbb{A}_k^\times$. Then, we can extend this character to one of $k^\times \backslash\mathbb{A}_k^\times$, which we denote by $\chi'$. At $k_{v_0}$, the character $\chi_{v_0}' \cdot \lambda^{-1}$ is unramified, hence a power of the absolute value character of $k_{v_0}^\times$. Twisting back $\chi'$ by a power of the absolute value id\`ele class character of $\mathbb{A}_k^\times$, we find the desired $\chi$.
\end{proof}

\subsection{}\label{begin:p0}
We now begin the proof when $\mathfrak{p}=0$. So we let $\pi$ be a supercuspidal representation of $G$. As $G$ is a product of ${\rm GL}_{n_i}(E_i)$'s, $\pi$ is a tensor product of cuspidal representations $\pi_i$ of ${\rm GL}_{n_i}(E_i)$; if $\sigma_i$ is the representation of $\mathcal{W}_{E_i}$ corresponding to $\pi_i$, there is an unramified character $\eta_i$ such that $\eta_i \sigma_i$ has finite image. It will be convenient to choose the extension $\widetilde{E}$ of $F$ in Lemma~\ref{GF:lemma} to contain, not only the $E_i$'s, but also the fixed field of ${\rm Ker}\,(\eta_i \sigma_i)$. As explained above, Lemma~\ref{GF:lemma} allows us to choose a number field $k$ and an extension $l$ so that ${\bf G}/F$ and $r$ lift to $k$. For each $i$, we let $l_i$ be the subfield of $\widetilde{E}$ corresponding to $E_i$ via the lemma. We shall assume that $\psi$ is the component at $v_0$ of a character $\Psi$ of $k \backslash \mathbb{A}_k$, which is enough by \S~\ref{123} 2).

It remains to lift $\pi$ appropriately to apply the reasoning sketched at the end of the introduction for the Asai cube case. The point is to globalize each $\sigma_i$ to a representation $\Sigma_i$ of $\mathcal{W}_{l_i}$, which corresponds via the global Langlands correspondence to a cuspidal automorphic representation $\Pi_i$ of ${\rm GL}_{n_i}(\mathbb{A}_{l_i})$. We then get from the $\Pi_i$'s a cuspidal automorphic representation $\Pi$ of ${\bf G}_k(\mathbb{A}_k)$, and from the $\Sigma_i$'s a corresponding global parameter $\Phi_\Sigma : \mathcal{W}_k \rightarrow {}^{\rm L}G_k$. As indicated in \S~\ref{sketch:proof}, we can then write the two global functional equations (vii) and ${\rm (vii)}_\mathscr{G}$ for a large enough finite set $S$ of places of $k$. Provided we have equality of the factors on both sides at places $v \in S \setminus \{ v_0 \}$, we get equality at $v_0$. Hence our desired equality \eqref{maineq:Asai}.

\subsection{}\label{dihedral:0}
Let us first treat the case where each $\sigma_i$ of dimension $2$ has dihedral image in the projective linear group --note that if the residual characteristic $p$ of $F$ is odd, we are necessarily in this dihedral situation \cite{We1974}. There is then for each $i$ such that $n_i=\dim \sigma_i = 2$ a quadratic separable extension $D_i$ of $E_i$ in $\overline{F}$ and a character $\chi_i$ of $D_i^\times$ such that $\sigma_i$ is induced from the character of $\mathcal{W}_{D_i}$ corresponding to $\chi_i$ via class field theory. By our choice of $\widetilde{E}$, $k$ and $l$, there is a quadratic extension $m_i$ of $l_i$ in $l$ giving $D_i/E_i$ at the place $v_0$.

By Lemma~\ref{GL1:glob} in ${\rm char}\,\mathfrak{p}=0$, we choose a character $\mathcal{X}_i$ of $\mathbb{A}_{m_i}^\times/m_i^\times$ with component $\chi_i$ at $v_0$, and unramified at other finite places. We take for $\Sigma_i$ the dimension $2$ representation of $\mathcal{W}_{l_i}$ induced from $\mathcal{X}_i$ (or rather the character of $\mathcal{W}_{m_i}$ corresponding to it). For each $i$ such that $n_i=\dim\sigma_i=1$, we use Lemma~\ref{GL1:glob} to globalize $\sigma_i$ to a character $\Sigma_i$ of $\mathcal{W}_{l_i}$, unramified at finite places different from $v_0$. By construction if $n_i=1$, and Jacquet-Langlands \cite{JaLa} if $n_i=2$, there is a cuspidal automorphic representation $\Pi_i$ of ${\rm GL}_{n_i}(\mathbb{A}_{l_i})$ corresponding to $\Sigma_i$ via the global Langlands correspondence. In this way, we get a cuspidal automorphic representation $\Pi$ of ${\bf G}_k(\mathbb{A}_k)$ and a global parameter $\Phi_\Pi : \mathcal{W}_k \rightarrow {}^{\rm L}G_k$ corresponding to it via the global Langlands correspondence, as globally assumed in \S~\ref{Galois:axioms}.

At a finite place $v \neq v_0$, $\Sigma_i$ is either unramified (if $n_i=1$ for example) or reducible, so that $\Pi_{i.v}$ is a principal series. In either case, the desired equality holds at $v$ by property (v) of \S~\ref{LS:axioms} or by the induction assumption of \S~\ref{123} 1). We conclude that equality holds at $v_0$ as well, which is what we wanted.

\subsection{}\label{nonD:0}
Let us now treat the case where $F$ has residue characteristic $2$ and some of the $\sigma_i$'s of dimension $2$ may have image in $A_4$ or $S_4$ in the projective linear group. We use Remark~\ref{rmk:0p}~(i) and choose $k$ (and $l$) to have only one place above $2$.

For the indices $i$ such that $\sigma_i$ has dimension $1$, we globalize as before, using Lemma~\ref{GL1:glob}. For the indices $i$ such that $\sigma_i$ has dimension $2$, as we have ensured, the representation $\eta_i \sigma_i$ of \S~\ref{begin:p0} has finite image and factorizes through ${\rm Gal}(l/k)$, we get a representation of $\mathcal{W}_{l_i}$ yielding $\eta_i\sigma_i$ at the place $v_0$, which we can twist by a power of the absolute value character to obtain a representation $\Sigma_i$ of $\mathcal{W}_{l_i}$ yielding $\sigma_i$ at the place $v_0$. The image of $\Sigma_i$ in the projective linear group is dihedral, $A_4$ or $S_4$ so there is by \cite{JaLa,La,Tu1978} a corresponding cuspidal automorphic representation $\Pi_i$ of ${\rm GL}_2(\mathbb{A}_{l_i})$. Since finite places of $l_i$ different from $v_0$ have odd residue characteristic, at such places $\Sigma_i$ is either reducible or dihedral. We then proceed as before in comparing two global functional equations for $\Sigma$ and $\Pi$: at finite places of $k$ in $S$ distinct from $v_0$, equality holds because we already treated the dihedral case, or by induction, and consequently it holds at $v_0$ as well.

\subsection{Case of $\mathfrak{p}>0$}\label{p>0}
We use the same reasoning by induction as in the characteristic zero case, so we can assume that $\pi$ is cuspidal. If all of the $n_i$'s are $1$, then $\bf G$ is a product of induced tori, and the result is known, more generally, for any torus \cite{HeLoFuture}. So we can assume $n_i=2$ for at least one $i$.

As in the characteristic zero case of \S~\ref{begin:p0}, we choose a Galois extension $\widetilde{E}$ of $F$ in $\overline{F}$, using Lemma~\ref{GF:lemma} to lift $\widetilde{E}/F$ to a global situation $l/k$. Following \S~\ref{begin:p0}, and using the same notation, we obtain representations $\sigma_i$, $\pi_i$. If for each $i$ such that $n_i=2$, $\sigma_i$ has dihedral image in the projective linear group, we can use Lemma~\ref{GL1:glob}, now in the characteristic $\mathfrak{p}>0$ case.

Following \S~\ref{dihedral:0} in the case of dihedral image for each $i$, such that $n_i=2$, we choose $D_i$ and $\chi_i$. To $D_i$ corresponds in $l$ a quadratic extension $m_i$ of $l_i$ and we extend $\chi_i$ to a character $\mathcal{X}_i$ of $\mathbb{A}_{m_i}^\times/m_i^\times$ unramified outside $v_0$ and some other place $v_1$, where it is tame. We choose $v_1$ to be above a place of $l_i$ split in $m_i$. So at that place of $l_i$, $\Sigma_i$ (obtained by inducing from $\mathcal{X}_i$) is reducible. The same reasoning as in \S~\ref{dihedral:0} then yields the result. We have now proved our result for $\mathfrak{p}>2$, since all the $\sigma_i$'s of dimension $2$ are dihedral.

If $\mathfrak{p}=2$ and one $\sigma_i$ of dimension $2$ has image $A_4$ or $S_4$ in the projective linear group, we use the result of Gabber and Katz \cite{Ka1986}. Namely, let $F$ be a non-archimedean local field of characteristic $p$, with residue field $\mathbb{F}_q$. We let $k = \mathbb{F}_q(t)$, and we choose an isomorphism $F \simeq \mathbb{F}_q (\!( t )\!)$, so that $F$ is the completion of $k$ at $0$. Then our Galois extension $\widetilde{E}/F$ can be globalized (uniquely, as it turns out) to a Galois extension $l$ of $k$, unramified outside $0$ and $\infty$, and tame at infinity. Now, following \S~\ref{nonD:0} we globalize $\sigma_i$ to $\Sigma_i$, which is unramified outside of $0$ and $\infty$ and tame at infinity; in particular, if $n_i=2$, $\Sigma_i$ is dihedral at $\infty$. Moreover, there is a cuspidal automorphic representation $\Pi_i$ of ${\rm GL}_{n_i}(l_i)$ corresponding to $\Sigma_i$, see Tunnell \cite{Tu1978}. We can then proceed as before, because at places other than $0$ and $\infty$ we are in an unramified situation, whereas at $\infty$ we have already treated the dihedral representations $\Sigma_{i,\infty}$ of dimension $2$. We can thus conclude equality at $0$.

\begin{remark}
The method of \cite{GaLo} would proceed a bit differently, rather producing first a globalization $\Pi_i$ of $\pi_i$ which is (at most) tamely ramified at $\infty$ and unramified outside $\left\{ 0, \infty \right\}$. Then invoking the proof by Drinfeld \cite{Dr1978} of the global Langlands conjecture for ${\rm GL}_2$ over $l_i$ to get an ($\ell$-adic) Galois representation $\Sigma_i$ associated to $\Pi_i$.
\end{remark}

\section{Consequences}\label{consequences}

In this section we first draw consequences of the main equality, on twisting with unramified characters first, and with highly ramified characters secondly, thus proving stability. We then recall the definition of $L$-functions and $\varepsilon$-factors, and we extend Theorem~\ref{GL2class:thm} from generic to general smooth irreducible representations $\pi$. Compatibility with the local Langlands correspondence for $L$-functions and $\varepsilon$-factors is an immediate corollary.

\subsection{Twisting with characters}\label{character}

Let $\bf G$ be our (pinned) quasi-split reductive group of ${\rm GL}_2$-kind defined over $F$. If $\bf G$ is the product of ${\rm Res}_{E_i/F}{\rm GL}_{n_i}$, $n_i \leq 2$, a character $\chi$ of $G = {\bf G}(F) = \prod_{i}{\rm GL}_{n_i}(E_i)$ is a product $\chi = \prod_i \chi_i \circ \det$ where each $\chi_i$ is a character of $E_i^\times$. Thus $\chi$ factorizes through the maximal quotient torus $Y$ of $G$, $Y = \prod_i E_i^\times = \prod_i {\bf Y}_i(F)$, where ${\bf Y}_i = {\rm Res}_{E_i/F}{\bf G}_m$. Dually, there is an embedding ${}^{\rm L}Y \rightarrow {}^{\rm L}G$ with $\widehat{\bf Y}$ going (isomorphically, in fact) to the center $Z_{\widehat{\bf G}}$ of $\widehat{\bf G}$.

If $\pi$ is a smooth irreducible representation of $G$, twisting $\pi$ by $\chi$ should correspond, via the local Langlands correspondence, to multiplying $\phi_\pi : \mathcal{W}_F \rightarrow {}^{\rm L}G$ by a map $\phi_\chi : \mathcal{W}_F \rightarrow Z_{\widehat{G}}$ determined by $\chi$ via the recipe of \cite{Bo1979}. And in fact it does: if $\tau_i$ is a smooth irreducible representation of ${\rm GL}_{n_i}(E_i)$ corresponding to a parameter $\phi_{\tau_i} : \mathcal{W}_{E_i} \rightarrow {\rm GL}_{n_i}(\mathbb{C})$, then twisting by $\tau_i$ by $\chi_i \circ \det$ corresponds to multiplying $\phi_{\tau_i}$ by the character $\eta_i$ of $\mathcal{W}_{E_i}$ corresponding to $\chi_i$ via class field theory. The $\eta_i$'s yield a parameter $\eta : \mathcal{W}_F \rightarrow {}^{\rm L}Y = \widehat{Y} \rtimes \mathcal{W}_F$ and twisting $\pi$ by $\chi$ indeed corresponds to twisting $\phi_\pi$ by the map given by the first component of $\eta$.

\subsection{Unramified characters} The group ${\rm Hom}_F({\bf G},{\bf G}_m)$ is a finitely generated free $\mathbb{Z}$-module, and the group of unramified characters $X_{\rm nr}(G)$ is isomorphic to a quotient of ${\rm Hom}_F({\bf G},{\bf G}_m) \otimes \mathbb{C}$: if $(\chi_1, \ldots, \chi_e)$ is a basis of ${\rm Hom}_F({\bf G},{\bf G}_m)$, then $\chi_1 \otimes s_1 + \cdots + \chi_e \otimes s_e$ corresponds to the unramified character $g \mapsto \prod_{i=1}^e \left| \chi_i(g) \right|_F^{s_i}$.

Let $r$ be an LS-representation of ${}^LG$, obtained via a pair $({\bf H},{\bf L})$ with $\bf H$ semisimple simply connected, and $\iota : {\bf L} \rightarrow {\bf G}$. Then ${\rm Hom}_F({\bf L},{\bf G}_m)$ has rank $1$, and a non-zero element is $\delta = \det({\rm Ad}_{\bf L}\vert_\mathfrak{n})$, where $\mathfrak{n}$ is the Lie algebra of the unipotent radical $\bf N$ of the standard parabolic subgroup $\bf P$ of $\bf G$ with Levi subgroup $\bf L$. We let $\tilde{\delta}$ be the unramified character corresponding to $\delta \otimes (1/\left\langle \delta,\alpha^\vee \right\rangle)$ if $\bf P$ is obtained by omitting the simple root $\alpha$.

Now $\iota$ induces a morphism ${\rm Hom}_F({\bf G},{\bf G}_m) \rightarrow {\rm Hom}_F({\bf L},{\bf G}_m)$ and an analogous morphism when tensoring with $\mathbb{C}$, which we still write as composition with $\iota$. If $\chi$ is an unramified character of $G$, the corresponding unramified character $\chi \circ \iota$ of $L$ has the form $\tilde{\delta}^{s_0}$ for $s_0 \in \mathbb{C}$ well defined up to $(2\pi i/\log q_F) \mathbb{Z}$. From the known property of unramified character twists for the LS-representation of ${}^{\rm L}L$ yielding $r$, Theorem~5.1 of \cite{LS}, we get:

\begin{enumerate}
\item[${\rm (viii)}$] \emph{(Twists by unramified characters).} Let $(F,\psi,{\bf G},\pi_i,r) \in \mathscr{A}_2(\mathfrak{p})$ be generic. Let $\chi$ be an unramified character of $G$, with $\chi \circ \iota = \tilde{\delta}^{s_0}$, then
   \begin{equation*}
      \gamma(s,\pi \otimes \chi,r,\psi) = \gamma(s+s_0,\pi,r,\psi).
   \end{equation*}
\end{enumerate}

\subsection{Highly ramified characters}\label{stability:prop} We shall take highly ramified characters in the following manner: we choose $\chi \in {\rm Hom}_F({\bf G},{\bf G}_m)$ yielding a morphism $G \rightarrow F^\times$, and compose it with a character $\lambda : F^\times \rightarrow \mathbb{C}^\times$. The following property is a corollary to Theorem~\ref{GL2class:thm}.

\begin{enumerate}
\item[${\rm (ix)}$] \emph{(Stability).} For $i = 1$ and $2$, let  $(F,\psi,{\bf G},\pi_i,r) \in \mathscr{A}_2(\mathfrak{p})$ be generic and such that their central characters are equal. Let $\chi \in {\rm Hom}_F({\bf G},{\bf G}_m)$ be such that $\chi \circ \iota \neq 0$. Then for all sufficiently ramified characters $\lambda : F^\times \rightarrow \mathbb{C}^\times$:
\begin{equation*}
   \gamma(s,\pi_1 \otimes (\lambda \circ \chi),r,\psi) = \gamma(s,\pi_2 \otimes (\lambda \circ \chi),r,\psi).
\end{equation*}
\end{enumerate}

\begin{proof}
We translate to the Weil group side, where a corresponding property has been proved by Deligne-Henniart. More precisely, it follows from \cite{DeHe} that if $\sigma_1$ and $\sigma_2$ are two representations of $\mathcal{W}_F$ with the same dimension and determinant, then for all sufficiently ramified characters $\eta$ of $\mathcal{W}_F$, we have
\begin{equation*}
   \gamma(s,\sigma_1 \otimes \eta,\psi) = 
   \gamma(s,\sigma_2 \otimes \eta,\psi).
\end{equation*}

If $\phi_{\pi_i} : \mathcal{W}_F \rightarrow {}^{\rm L}G$ is the parameter of $\pi_i$, the parameter for $\pi_i \otimes(\lambda \circ \chi)$ is $\phi_{\pi_i} \cdot \phi_{\lambda \circ \chi}$ where $\phi_{\lambda \circ \chi}$ is described in \S~\ref{character}: explicitly $\chi : {\bf G} \rightarrow {\bf G}_m$ yields an L-map ${}^{\rm L}\chi : {}^{\rm L}{G}_m \rightarrow {}^{\rm L}G$ and ${}^{\rm L}\chi(\mathbb{C}^\times)$ is contained in the $\mathcal{W}_F$-fixed part of the center $Z_{\widehat{G}}$ of $\widehat{G}$; $\phi_{\lambda \circ \chi}$ is the map $\mathcal{W}_F \rightarrow {}^{\rm L}G$ obtained by composing ${}^{\rm L}\chi$ with the character $\mathcal{W}_F \rightarrow \mathbb{C}^\times$ corresponding to $\lambda$ via class field theory. Now $r = \rho \circ {}^{\rm L}\iota$ where $\rho$ is an irreducible representation of ${}^{\rm L}L$ and ${}^{\rm L}\iota : {}^{\rm L}G \rightarrow {}^{\rm L}L$ is dual to $\iota : {\bf L} \rightarrow {\bf G}$.

Because $\iota$ induces an isogeny on derived subgroups ${}^{\rm L}\iota(Z_{\widehat{G}})$ is contained in the center $Z_{\widehat{L}}$ of $\widehat{L}$, and ${}^{\rm L}\iota \circ \phi_{\lambda \circ \chi}$ is a character with values in $Z_{\widehat{L}}^{\mathcal{W}_F}$. So that
\begin{equation}
   r(\phi_{\pi_i} \cdot \phi_{\lambda \circ \chi}) = 
   [ \rho \circ {}^{\rm L}\iota(\phi_{\pi_i}) ] \cdot 
   \omega_\rho({}^{\rm L}\iota \circ \phi_{\lambda \circ \chi}),
\end{equation}
where $\omega_\rho$ is the central character of $\rho$. The assumption on $\chi$ implies that $\omega_\rho({}^{\rm L}\iota\circ{}^{\rm L}\chi)$ is a non-trivial (algebraic) character of $\mathbb{C}^\times$, and composing with sufficiently ramified characters of $\mathcal{W}_F$ still gives sufficiently ramified characters. It remains to verify that $r(\phi_{\pi_1})$ and $r(\phi_{\pi_2})$ have the same determinant: the above result of \cite{DeHe}, then gives the result.

To prove $\det r(\phi_{\pi_1}) = \det r(\phi_{\pi_2})$, we use a property of the local Langlands correspondence: if $\pi$ is a smooth irreducible representation of ${\rm GL}_n(E)$ (where $E$ is a finite extension of $F$ in $\overline{F}$) and $\phi_\pi : \mathcal{W}_E \rightarrow {\rm GL}_m(\mathbb{C})$ is its L-parameter, then $\det \circ\,\phi_\pi : \mathcal{W}_E \rightarrow \mathbb{C}^\times$ corresponds to the central character $\omega_\pi$ of $\pi$ via class field theory.

Write ${\bf G} = \prod_j {\bf G}_j$, where ${\bf G}_j = {\rm Res}_{E_j/F}{\rm GL}_{n_j}$, $n_j \leq 2$, and accordingly $\pi_i = \otimes_j \pi_{i,j}$, where $\pi_{i,j}$ is a smooth irreducible representation of $G_j = {\rm GL}_{n_j}(E_j)$. If $\phi_{i,j} : \mathcal{W}_{E_j} \rightarrow {\rm GL}_{n_j}(\mathbb{C})$ is the parameter of $\pi_{i,j}$, then $\phi_{\pi_i} : \mathcal{W}_F \rightarrow {}^{\rm L}G$ is obtained as follows: we see $\phi_{i,j}$ as a morphism $\mathcal{W}_{E_i} \rightarrow {}^{\rm L}{\rm GL}_{n_j}/E_j = {\rm GL}_{n_j}(\mathbb{C}) \ltimes \mathcal{W}_{E_j}$ (given by the identity on the second component), giving an induced morphism $\phi_{i,j}' : \mathcal{W}_F \rightarrow {}^{\rm L}{G}_{j} = \widehat{G}_j \ltimes \mathcal{W}_F$ and $\phi_\pi$ is the ``product'' of these morphisms in the sense that on the second component $\widehat{G} = \prod_j \widehat{G_j}$ it is the product of the $\phi_{i,j}'$. Now $\det \circ\,r$ is a $\mathcal{W}_F$-invariant algebraic character of $\widehat{G}$, which on each component comes from some power of the determinant character ${\rm GL}_{n_j}(\mathbb{C})$. Thus $\det r(\phi_{\pi_i})$ is a character of $\mathcal{W}_F$, where we can expand the determinant and find that it is given in an explicit way by the characters $\det \circ\,\phi_{i,j}$ of $\mathcal{W}_{E_j}$, i.e., by the central characters $\omega_{\pi_{i,j}}$. If $\omega_{\pi_1} = \omega_{\pi_2}$, then certainly we have $\omega_{\pi_{1,j}} = \omega_{\pi_{2,j}}$ for each $j$, and consequently $\det r (\phi_{\pi_1}) = \det r (\phi_{\pi_2})$.
\end{proof}


\subsection{$L$-functions and $\varepsilon$-factors}\label{L3:axioms} We first recall how to obtain $L$-funtions and $\varepsilon$-factors from $\gamma$-factors. And, with the same construction, we extend the definition of $\gamma$-factors from generic to general smooth irreducible representations. We begin by first recording the local functional equation for generic $\gamma$-factors, which follows either from uniqueness or from Theorem~\ref{GL2class:thm} and the corresponding property for Galois $\gamma$-factors.

\begin{enumerate}
\item[${\rm (x)}$] \emph{(Local functional equation).} Let $(F,\psi,{\bf G},\pi_i,r) \in \mathscr{A}_2(\mathfrak{p})$ be generic, then
   \begin{equation*}
      \gamma(s,\pi,r,\psi) \gamma(1-s,\tilde{\pi},r,\overline{\psi}) = 1.
   \end{equation*}
\end{enumerate}

We say that $(F,\psi,{\bf G},\pi,r) \in \mathscr{A}_2(\mathfrak{p})$ is tempered if $\pi$ is a tempered representation of $G$. Observe that for $\bf G$ of ${\rm GL}_2$-kind, a tempered representation is also generic.

The following property is the tempered $L$-function conjecture, originally stated in \cite{Sh1990} and proved for Langlands-Shahidi $L$-functions in \cite{HeOp2013}. Furthermore, we observe that for $\bf G$ of ${\rm GL}_2$-kind (even ${\rm GL}_n$-kind), given a tempered representation $\pi$ its L-parameter $\phi_\pi$ is bounded. Then if $\tau$ is the generic component of the smooth representation $\pi \circ \iota$ of $L$, then $\phi_{\tau}$ is also bounded. We thus obtain an alternative proof of the tempered $L$-function conjecture for $\bf G$ of ${\rm GL}_2$-kind, which now follows from the Galois side, in view of Theorem~\ref{GL2class:thm}.

\begin{enumerate}
\item[${\rm (xi)}$] \emph{(Tempered $L$-functions).} For $(F,\psi,{\bf G},\pi,r) \in \mathscr{A}_2(\mathfrak{p})$ tempered, let $P_{\pi}(t)$ be the unique polynomial with $P_{\pi}(0) = 1$ and such that $P_{\pi}(q_F^{-s})$ is the numerator of $\gamma(s,\pi,r,\psi)$. Then
   \begin{equation*}
      L(s,\pi,r) = \dfrac{1}{P_{\pi}(q_F^{-s})}.
   \end{equation*}
is holomorphic and non-zero for $\Re(s) > 0$.
\end{enumerate}

\noindent The above two properties lead us to the following well defined notion of $\varepsilon$-factors.

\begin{enumerate}
\item[${\rm (xii)}$] \emph{(Tempered $\varepsilon$-factors).} Let $(F,\psi,{\bf G},\pi_i,r) \in \mathscr{A}_2(\mathfrak{p})$ be tempered, then
   \begin{equation*}
      \varepsilon(s,\pi,r,\psi) = 
      \gamma(s,\pi,r,\psi) 
      \dfrac{L(s,\pi,r)}{L(1-s,\tilde{\pi},r)}
   \end{equation*}
is a monomial in $q_F^{-s}$.
\end{enumerate}

If we start with a tempered (generic) representation, the dependence of the $L$-factor is holomorphic when we twist by unramified characters of $G$. This allows us to use the Langlands classification, which we now state, in order to address the general case.

\begin{enumerate}
\item[${\rm (xiv)}$] \emph{(Langlands classification).} Let $(F,\psi,{\bf G},\pi,r) \in \mathscr{A}_2(\mathfrak{p})$. Let $\bf P$ be a parabolic subgroup of $\bf G$, with Levi subgroup $\bf M$, and such that $\pi$ is the unique quotient of
    \begin{equation*}
      {\rm ind}_P^G \rho \otimes \chi,
   \end{equation*}
where $\rho$ is a tempered representation of $M$ and $\chi$ is an unramified character in the Langlands situation. Let $\phi_\pi$ be the $L$-parameter attached to $\pi$ via the local Langlands correspondence; and, let $\phi_\tau$ be the L-parameter of $\tau = \tau_0 \otimes \chi$. Then
\begin{equation*}
   \phi_\pi: \mathcal{W}_F' \xrightarrow{\phi_\tau} {}^{\rm L}M 
   \longrightarrow {}^{\rm L}G.
\end{equation*}
\end{enumerate}

It is via Langlands classification that we extend the definition of $\gamma$-factors from generic to general smooth irreducible representations; in addition to passing from tempered $L$-functions and $\varepsilon$-factors to the general setting. We observe that, while we do obtain multiplicativity for $\gamma$-factors, we do not have multiplicativity of local $L$-functions and root numbers in general. The construction of $L$-functions and $\varepsilon$-factors from $\gamma$-factors follows the idea of Shahidi \cite{Sh1990}.

Now, let $(F,\psi,{\bf G},\pi,r) \in \mathscr{A}_2(\mathfrak{p})$ be tempered (hence generic). A tempered representation $\pi$ corresponds to a bounded $L$-parameter $\phi_\pi$, for which its $L$-function is indeed given by the numerator of the corresponding Galois $\gamma$-factor. For tempered representations, Galois $\gamma$-factors, $L$-functions and $\varepsilon$-factors behave well under all unramified twists.

Given a general $(F,\psi,{\bf G},\pi,r) \in \mathscr{A}_2(\mathfrak{p})$, local Galois factors are compatible with Langlands classification. We thus conclude the proof of Theorem~\ref{GL2class:thm} and obtain the following corollary.

\begin{corollary}\label{L3:thm}
Let $(F,\psi,{\bf G},\pi,r) \in \mathscr{A}_2(\mathfrak{p})$ and let $\phi_\pi$ be the $L$-parameter such that $\pi \leftrightarrow \phi_\pi$ are related via the local Langlands correspondence. Then
\begin{align*}
   \gamma(s,\pi,r,\psi) &= \gamma(s,r \circ \phi_\pi,\psi), \\
   L(s,\pi,r) &= L(s,r \circ \phi_\pi), \\ 
   \varepsilon(s,\pi,r,\psi) &= \varepsilon(s,r \circ \phi_\pi,\psi).
\end{align*}
\end{corollary}

\section{Asai cube representation}\label{AsaiL}

In this section we show in detail that the Asai cube factors can indeed be obtained as LS-factors for a group of ${\rm GL}_2$-kind; applying Theorem~\ref{GL2class:thm} then gives \eqref{maineq:Asai}. Let $F$ be a local field or a global one; we fix a separable algebraic closure $\overline{F}$ of $F$ and let $\mathcal{W}_F$ be the Weil group of $\overline{F}/F$. We are given a cubic separable extension $E$ of $F$ in $\overline{F}$, with Galois closure $\widetilde{E}$, and we consider the group ${\bf G} = {\rm Res}_{E/F}{\rm GL}_2$.

The $L$-group of $\bf G$ is the semi-direct product ${}^{\rm L}G = {\rm GL}_2(\mathbb{C})^J \rtimes \mathcal{W}_F$, where $J = {\rm Hom}_F(E,\bar{F})$ and $\mathcal{W}_F$ acts on ${\rm GL}_2(\mathbb{C})^J$ via its natural action on $J$. The obvious $8$-dimensional representation of ${\rm GL}_2(\mathbb{C})^J$ on $\otimes_{j \in J} \mathbb{C}^2$ extends to a representation ${}^\otimes{\rm I}$ of ${}^{\rm L}G$, where again $\mathcal{W}_F$ acts on $\otimes_{j \in J} \mathbb{C}^2$ by permuting the factors $\mathbb{C}^2$ via its natural action on $J$. When $E$ is a cubic extension of $F$ in $\bar{F}$, we can consider the Weil group $\mathcal{W}_E$ of $\bar{F}$ over $E$ as a subgroup of index $3$ of $\mathcal{W}_F$, and ${}^\otimes{\rm I}$ is obtained by tensor induction from the representation of ${\rm GL}_2(\mathbb{C})^J \rtimes \mathcal{W}_E$, where ${\rm GL}_2(\mathbb{C})^J$ acts via its $j$-component, and $\mathcal{W}_E$ acts trivially. We call ${}^\otimes{\rm I}$ the \emph{Asai cube representation} of ${}^{\rm L}G$.

\subsection{} All our reductive groups over $F$ are quasi-split and pinned, i.e., equipped with a Borel subgroup and a Levi subgroup (in fact, a torus) of that Borel subgroup, both defined over $F$, and a $\mathcal{W}_F$-equivariant choice of isomorphisms of the root subgroup corresponding to a simple root (over $\overline{F}$) with the additive group ${\bf G}_a$ \cite{Bo1979}. For ${\rm GL}_2$, we take the standard pinning and $\bf G$ is equipped with the induced one [\emph{loc.\,cit.}].

To produce the $\gamma$-factor we need a (pinned) quasi-split reductive group $\bf H$ over $F$, which we take semisimple simply connected of type $D_4$, with triviality given by $E$ (see below). We consider the parabolic subgroup $\bf P$ of $\bf H$ (containing the fixed Borel subgroup) obtained by omitting the central root of the Dynkin diagram, and let $\bf L$ be its Levi subgroup containing the fixed maximal torus. We shall produce a morphism of pinned reductive groups $\iota : {\bf L} \rightarrow {\bf G}$ defined over $F$ and inducing a central isogeny on the derived subgroups.

The pinning of $\bf H$, gives a based root datum $\mathcal{R}_H$ with an action of $\mathcal{W}_F$ \cite{Bo1979} factoring through the Galois group of $\widetilde{E}$ over $F$. Since $\mathcal{R}_H$ with its action of $\mathcal{W}_F$ determines $\bf H$ up to unique isomorphism, we shall work only with $\mathcal{R}_H$ (or rather its dual root datum $\mathcal{R}_H^\vee$). Similarly $\iota$ induces a $\mathcal{W}_F$-equivariant morphism of based root data $\iota_\mathcal{R} : \mathcal{R}_L \rightarrow \mathcal{R}_G$, by which $\iota$ is uniquely determined, so again we only need describe $\iota_\mathcal{R}$.

\subsection{} On the dual side, we choose a complex (pinned) reductive group $\widehat{\bf H}$, with based root datum the dual $\mathcal{R}_H^\vee$ of $\mathcal{R}_H$; the natural action of $\mathcal{W}_F$ on $\mathcal{R}_H^\vee$ produces an L-group ${}^{\rm L}H = \widehat{H} \ltimes \mathcal{W}_F$.

We let $\widehat{\bf P}$ be the parabolic subgroup of $\widehat{\bf H}$ corresponding to $\bf P$ (and containing the underlying Borel subgroup). We let $\widehat{\bf L}$ be the Levi subgroup of $\widehat{\bf P}$ containing the underlying maximal torus $\widehat{\bf T}$, and write $\widehat{\bf N}$ for the unipotent radical of $\widehat{\bf P}$. Then $\widehat{\bf L}$ has based root datum $\mathcal{R}_L^\vee$ and the L-group ${}^{\rm L}L = \widehat{H} \ltimes \mathcal{W}_F$ is a subgroup of ${}^{\rm L}H$; it acts naturally on the Lie algebra ${}^{\rm L}\mathfrak{n}$ of $\widehat{\bf N}$, and it is that action that we use to construct our $\gamma$-factors (see \S~\ref{adjoint} below --could define things in \S~2).

\subsection{} We now specify $\mathcal{R}_H$ explicitly. Although we rather work with the dual root datum $(X,\Phi,X^\vee,\Phi^\vee)$, where $X$ is the group of characters of $\widehat{\bf T}$ and $\Phi$ is the set of roots of $\widehat{\bf T}$ in $\widehat{\bf H}$; the group $X^\vee$ is the group of cocharacters of $\widehat{\bf T}$ and $\Phi^\vee$ is the set of coroots, equipped with a bijection $\alpha \mapsto \alpha^\vee$ of $\Phi$ onto $\Phi^\vee$. Finally, the group $\mathcal{W}_F$ acts linearly on $X$ (via a finite quotient), and that action preserves $\Phi$; the dual action on $X^\vee$ preserves $\Phi^\vee$, compatibly with the bijection $\alpha \mapsto \alpha^\vee$.

Our group ${}^{\rm L}H$ has type $D_4$ and to describe its root datum we use the notation of Bourbaki \cite{Bou}. However, we prefer to separate the roles of $X$ and $X^\vee$ (and not identify them via some Killing form), so that we let $X$ be the set of vectors in $V = \mathbb{R}^4$ with integer coordinates adding to an even number; we write $(e_1, \ldots, e_4)$ for the canonical basis of $V$, and choose $\alpha_1 = e_1 - e_2$, $\alpha_2 = e_2 - e_3$, $\alpha_3 = e_3 - e_4$ and $\alpha_4 = e_3 + e_4$ as a basis of $\Phi$. The vector space $V^\vee$ dual to $V$ has the basis $(e_1^\vee, \ldots, e_4^\vee)$ dual to $(e_1, \ldots, e_4)$; the simple coroots are $\alpha^\vee_1 = e^\vee_1 - e^\vee_2$, $\alpha^\vee_2 = e^\vee_2 - e^\vee_3$, $\alpha^\vee_3 = e^\vee_3 - e^\vee_4$ and $\alpha^\vee_4 = e^\vee_3 + e^\vee_4$; the lattice $X^\vee$ in $V^\vee$ is generated by $\frac{1}{2}(e_1^\vee + e_2^\vee + e_3^\vee + e_4^\vee)$ and $e_2^\vee$, $e_3^\vee$, $e_4^\vee$.

Writing $\mathfrak{S}$ for the group of permutations of $\left\{ 1,3,4 \right\}$, we have a natural action of $\mathfrak{S}$ on $V$ preserving $X$, fixing $\alpha_2$ and permuting $\alpha_1$, $\alpha_3$, $\alpha_4$ according to the indices.

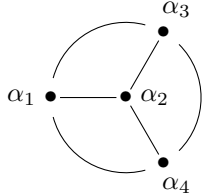
\begin{center}
\begin{tikzpicture}
   \draw[circle] (180:1) node {$\bullet$};
   \draw[circle] (180:1) node[left] {$\alpha_1$};
   \draw[circle] (195:1) arc (195:285:1);
   \draw[circle] (300:1) node {$\bullet$};
   \draw[circle] (300:1.35) node {$\alpha_4$};
   \draw[circle] (315:1) arc (315:405:1);
   \draw[circle] (60:1) node {$\bullet$};
   \draw[circle] (60:1.35) node {$\alpha_3$};
   \draw[circle] (75:1) arc (75:165:1);
   \draw[circle] (0:0) node {$\bullet$} node[right] {$\alpha_2$};
   \draw[circle] (60:.125) -- (60:.875);
   \draw[circle] (180:.125) -- (180:.875);
   \draw[circle] (300:.125) -- (300:.875);
\end{tikzpicture}
\end{center}

\noindent Any group homomorphism $\mathcal{W}_F \rightarrow \mathfrak{S}$ then gives a group root datum with action of $\mathcal{W}_F$, determining a pinned reductive groups $\bf H$ over $F$. If $E$ is our fixed cubic extension of $F$ in $\bar{F}$, any identification of $J = {\rm Hom}_F(E,\bar{F})$ with $\left\{ 1, 3, 4 \right\}$ gives a homomorphism $\mathcal{W}_F \rightarrow \mathfrak{S}$ producing the group ${}^{\rm L}H$ we seek.

\begin{remark}\label{GFasai}
\emph{If we start with a global field $k$ and a cubic separable extension $l$ of $k$, we are producing groups ${\bf H}/k$ and ${}^{\rm L}H_k$. If $v$ is a place of $k$, the root datum for the $L$-group ${}^{\rm L}H_{k_v}$ of ${\bf H} \otimes_{k} k_v$ is obtained from the root datum of ${}^{\rm L}H_k$ by composing the action of $\mathcal{W}_k$ with the homomorphism $\mathcal{W}_{k_v} \rightarrow \mathcal{W}_k$ coming from the completion (such a homomorphism depends on an isomorphism of $\bar{k}$ with the algebraic closure of $k$ in $\bar{k}_v$, but changing it changes ${}^{\rm L}H_{k_v}$ to an isomorphic group). Note that even if the homomorphism $\mathcal{W}_k \rightarrow \mathfrak{S}$ is surjective, the local homomorphism $\mathcal{W}_{k_v} \rightarrow \mathfrak{S}$ might not be. For example, at a split place $v$ of a global cubic extension $l/k$ we have $l_v \simeq k_v \times k_v \times k_v$; the local homomorphism $\mathcal{W}_{k_v} \rightarrow \mathfrak{S}$ is trivial. This explains why we cannot restricted to the cases where $E/F$ is always a cubic extension, and why it is better to deal with a general situation of ${\rm GL}_2$-kind.}
\end{remark}

\subsection{Adjoint representation}\label{adjoint} The roots of $\widehat{\bf T}$ in $\widehat{\bf L}$ are the roots in $\Phi$ which are linear combinations of $\alpha_1$, $\alpha_3$, $\alpha_4$, whereas the roots of $\widehat{\bf T}$ in ${}^{\rm L}\mathfrak{n}$ are the (positive) roots in $\Phi$ where $\alpha_2$ appears with a positive coefficient. The adjoint representation of $\widehat{\bf L}$ on ${}^{\rm L}\mathfrak{n}$ has two irreducible components $r_i$, $i = 1$, $2$, and the corresponding roots of $\widehat{\bf T}$ are the roots in $\Phi$ where $\alpha_2$ appears with coefficient $i$. We use the LS-representation of ${}^{\rm L}G$ coming from $r_1$ via a morphism $\iota : {}^{\rm L}G \rightarrow {}^{\rm L}L$ that we now construct.

Now we relate ${}^{\rm L}L$ and ${}^{\rm L}G$, and $r_\mathcal{A}$ with ${}^\otimes{\rm I}$. With the chosen identification of $J$ with $\left\{ 1, 3, 4 \right\}$, ${}^{\rm L}G$ becomes ${\rm GL}_2(\mathbb{C})^{\left\{ 1,3,4 \right\}} \rtimes \mathcal{W}_F$; let us describe the corresponding root datum and relate it to the root datum for ${}^{\rm L}L$. For a root datum $(Y,\Psi,Y^\vee,\Psi^\vee)$ for ${}^{\rm L}G$, we can take
\begin{equation*}
   Y = Y_1 \oplus Y_2 \oplus Y_3 \text{ with } Y_i = \mathbb{Z} e_i \oplus \mathbb{Z} f_i,
\end{equation*}
where $\alpha_i = e_i - f_i$ for $i = 1, 3, 4$ as simple roots; similarly
\begin{equation*}
   Y^\vee = Y^\vee_1 \oplus Y^\vee_2 \oplus Y^\vee_3 \text{ with } Y^\vee_i = \mathbb{Z} e^\vee_i \oplus \mathbb{Z} f^\vee_i,
\end{equation*}
where $\alpha^\vee_i = e^\vee_i - f^\vee_i$ for $i = 1,3,4$ -- the duality is the obvious one, and $\mathcal{W}_F$ acts via its action on $\left\{ 1,3,4 \right\}$.

Consider the quotient ${}^{\rm L}\overline{G}$ of ${}^{\rm L}G$ by the central subgroup made out of central elements $(x_1,x_3,x_4)$ in ${\rm GL}_2(\mathbb{C})^{\left\{ 1,3,4 \right\}}$ such that $x_1 x_3 x_4 = 1$; the corresponding root datum is $(Z,\Psi,\bar{Z}^\vee,\overline{\Psi}^\vee)$, where $Z$ is the sublattice of $Y$ of elements $a_1e_1 + b_1f_1 + a_3e_3 + b_3d_3 + a_4e_4 + b_4f_4$ such that $a_1 + b_1 = a_3 + b_3 = a_4 + b_4$; then $\bar{Z}^\vee$ is the corresponding quotient of $Y^\vee$ and $\overline{\Psi}^\vee$ is the image of $\Psi^\vee$ in $\bar{Z}^\vee$. The action of $\mathcal{W}_F$ is, again, obtained via the action on $\left\{ 1,3,4 \right\}$. One verifies immediately that the root datum $(Z,\Psi,\bar{Z}^\vee,\overline{\Psi})$ is isomorphic to the one for ${}^{\rm L}L$ by sending $\alpha_i$ to $\alpha_i$, for $i = 1$, $3$, $4$, and $f_1 + f_2 + f_3$ to $\alpha_2$, so that dually $\bar{\alpha}_i^\vee$ in $\bar{Z}^\vee$ is sent to $\alpha_i^\vee$, for $i =1$, $3$, $4$, wheras the image of $e_1^\vee + f_1^\vee$ in $\bar{Z}^\vee$ (which is the same as the image of $e_2^\vee + f_2^\vee$ or $e_3^\vee  + f_3^\vee$), is sent to $\alpha_1^\vee + \alpha_3^\vee + \alpha_4^\vee - 2 \alpha_2^\vee$ (so that $e_1^\vee$ is sent to $\alpha_1^\vee - \alpha_2^\vee + \frac{1}{2}(\alpha_3^\vee + \alpha_4^\vee)$).

It follows that we can choose an isomorphism $\varphi$ of ${}^{\rm L}\overline{G}$ onto ${}^{\rm L}L$, compatible with the isomorphism of root data just described. Since the representation $r$ of $\widehat{\bf L}$ on ${}^{\rm L}\mathfrak{n}$ has weights $\alpha_2$, $\alpha_1 + \alpha_2$, $\alpha_2 + \alpha_3$, $\alpha_2 + \alpha_4$, $\alpha_1 + \alpha_2 + \alpha_3$, $\alpha_1 + \alpha_2 + \alpha_4$, $\alpha_2 + \alpha_3 + \alpha_4$, $\alpha_1 + \alpha_2 + \alpha_3 + \alpha_4$, we see that through $\widehat{\bf G} \rightarrow \overline{\widehat{\bf G}} \rightarrow \widehat{\bf L}$, $r$ gives rise to the tensor product representation of $\widehat{\bf G} = {\rm GL}_2(\mathbb{C})^{\left\{ 1,3,4 \right\}}$, $\left\{ g_1, g_3, g_4 \right\} \mapsto g_1 \otimes g_3 \otimes g_4$. That identification even extends to ${}^{\rm L}G$ and its action via $r_\mathcal{A}$: indeed in the representation of ${}^{\rm L}L$ on ${}^{\rm L}\mathfrak{n}$ we can choose bases for the root subspaces so that the action of $\mathfrak{S}$ (hence $\mathcal{W}_F$) on these vectors is given by its action on the roots via the permutation of $\left\{ 1,3,4 \right\}$. It is then clear that via ${}^{\rm L}G \rightarrow {}^{\rm L}L$, $r_1$ does indeed give ${}^\otimes{\rm I}$.

\subsection{Remark} It appears to be impossible to construct a quasi-split group ${\bf H}'$ over $F$, with a Levi subgroup ${\bf L}'$ isomorphic to $\bf G$ and giving rise to the representation $r$ by the LS method; this is where the extra datum of $\iota: {\bf L} \rightarrow {\bf G}$ is important. We study this issue in detail in \cite{HeLoFuture}.

\medskip

\noindent{\sc \Small Guy Henniart, Institut Universitaire de France et Universit\'e Paris-Sud, Laboratoire de Math\'ematiques d'Orsay, CNRS, Orsay cedex F-91405, France}\\
\emph{\Small E-mail address: }\texttt{\Small Guy.Henniart@math.u-psud.fr}

\noindent{\sc \Small Luis Alberto Lomel\'i, Instituto de Matem\'aticas, Pontificia Universidad Cat\'olica de Valpara\'iso, Blanco Viel 596, Cerro Bar\'on, Valpara\'iso, Chile}\\
\emph{\Small E-mail address: }\texttt{\Small Luis.Lomeli@pucv.cl}

\end{document}